\documentclass[11pt]{article}

\usepackage{lineno}
\usepackage[pagebackref]{hyperref}
\modulolinenumbers[5]
\usepackage[numbers]{natbib}

\usepackage[margin=0pt,position=t]{subcaption}
\usepackage{fullpage}

\usepackage[latin1]{inputenc}
\usepackage[T1]{fontenc}

\usepackage{xspace}
\usepackage{amsmath,amssymb,amsthm}

\usepackage{mathtools}
\usepackage{xcolor,colortbl}
\definecolor{Gray}{gray}{0.85}
\definecolor{LightCyan}{rgb}{0.88,1,1}

\usepackage{xcolor}
\definecolor{dargray}{rgb}{0.18, 0.18, 0.18}
\definecolor{darkgreen}{rgb}{0.01,0.6,0.1}
\definecolor{lightrose}{rgb}{0.996,0.75,0.793}
\definecolor{rose}{cmyk}{0.75, 0.75, 0,0}
\definecolor{winered}{rgb}{0.6,0.1,0.1}
\definecolor{lightyellow}{rgb}{1, 1, 0.6}
\definecolor{transparent}{rgb}{1,1,1}
\definecolor{lightlightgray}{rgb}{0.88, 0.88, 0.88}
\definecolor{lightgray}{rgb}{0.8, 0.8, 0.8}
\definecolor{lightblue}{rgb}{0.527,0.805,0.977}
\definecolor{lightgreen}{rgb}{.74,1,0}

\hypersetup{menucolor=orange!40!black,filecolor=winered,urlcolor=green!40!black, pdfencoding=auto, psdextra,
colorlinks=true,citecolor=green!60!black,linkcolor=winered,
pagebackref=true}

\newcolumntype{a}{>{\columncolor{Gray}}c}
\newcolumntype{b}{>{\columncolor{white}}c}

\newcommand{\clR}[1]{\textcolor{black}{#1}}
\newcommand{\clB}[1]{\textcolor{black}{#1}}
\newcommand{\clG}[1]{\textcolor{black}{#1}}
\newcommand{\clO}[1]{\textcolor{black}{#1}}

\usepackage{hhline}

\usepackage{tikz}
\usetikzlibrary{decorations,arrows,topaths,backgrounds,shapes,positioning,fit,calc,decorations.pathreplacing,patterns,intersections,matrix}

\usepackage{pgfplots}
\pgfplotsset{compat=1.3}

\usepackage{paralist}
\usepackage{array}

\newtheorem{corollary}{Corollary}
\newtheorem{lemma}{Lemma}

\newtheorem{proposition}{Proposition}

\newtheorem{theorem}{Theorem}
\newtheorem{claim}{Claim}

\usepackage[nameinlink]{cleveref}
\crefname{table}{Table}{Tables}
\crefname{figure}{Figure}{Figures}
\crefname{theorem}{Theorem}{Theorems}
\crefname{corollary}{Corollary}{Corollaries}
\crefname{observation}{Observation}{Observations}
\crefname{lemma}{Lemma}{Lemmas}
\crefname{reduction}{Reduction}{Reductions}
\crefname{construction}{Construction}{Constructions}
\crefname{subsection}{Section}{Sections}
\crefname{section}{Section}{Sections}
\crefname{proposition}{Proposition}{Propositions}
\crefname{claim}{Claim}{Claims}

\theoremstyle{definition}
\newtheorem{definition}{Definition}
\newtheorem{example}{Example}
\crefname{example}{Example}{Examples}
\crefname{definition}{Definition}{Definitions}

\newcommand{\ppp}{{\mathcal{P}}}
\newcommand{\vvv}{{\mathcal{V}}}

\newcommand{\prefs}{(\pref_1,\pref_2,\ldots, \pref_n)}

\newcommand{\dde}[1][]{{\ifthenelse{\equal{#1}{}}{$d$}{$#1$}-dimensional Euclidean}\xspace}
\newcommand{\dder}[1][]{{\ifthenelse{\equal{#1}{}}{$d$}{$#1$}-dimensional Euclidean representation}\xspace}

\newcommand{\dEuclid}[1][]{{\ifthenelse{\equal{#1}{}}{$d$}{$#1$}-Euclidean}\xspace}

\newcommand{\pref}{\ensuremath{\succ}}

\newcommand{\defvoter}{\text{voter}}

\newcommand{\topp}{\mathsf{top}}
\newcommand{\peak}{\mathsf{peak}}
\newcommand{\diffpairs}{\mathsf{diff}\text{-}\mathsf{pairs}}
\newcommand{\SSYT}{\mathsf{SSYT}}
\newcommand{\pos}{\mathsf{pos}}

\usepackage{authblk}

\usepackage{multicol}
\usepackage{multirow}

\newcommand{\mytitle}{%
On the Number of Single-Peaked Narcissistic 
or Single-Crossing Narcissistic Preference Profiles\thanks{To appear in Discrete Mathematics.\newline Correspondence to: Jiehua Chen, Dept.\ Industrial Engineering and Management Ben Gurion University of the Negev Beer Sheva, 84105, Israel P.O.B 653}}

\author[1]{Jiehua Chen}

\author[2]{Ugo P.\ Finnendahl}

\affil[1]{Ben-Gurion University of the Negev, Israel 
  \texttt{jiehua.chen2@gmail.com}}

\affil[2]{Institut f\"ur Softwaretechnik und Theoretische Informatik,
  TU Berlin, Berlin, Germany
  \texttt{ugo.p.finnendahl@campus.tu-berlin.de}}
\date{}

\begin{document}

\sloppy
\allowdisplaybreaks

\title{\mytitle}

\maketitle
\begin{abstract}
  We investigate preference profiles for a set~$\vvv$ of voters,
  where each voter~$i$ has 
  a preference order~$\pref_i$ 
  on a finite set~$A$ of alternatives (that is, a linear order on~$A$) 
  such that for each two alternatives~$a,b\in A$, voter~$i$ \emph{prefers} $a$ to $b$ if $a\pref_i b$.
  Such a profile is \emph{narcissistic} if each alternative~$a$ is preferred the \emph{most} by at least one voter. 
  It is \emph{single-peaked} if there is a linear order~$\triangleright^{\text{sp}}$ on the alternatives 
  such that each voter's preferences on the alternatives along the order~$\triangleright^{\text{sp}}$ 
  are either strictly increasing, or strictly decreasing, or first strictly increasing and then strictly decreasing. 
  It is \emph{single-crossing} 
  if there is a linear order~$\triangleright^{\text{sc}}$ on the voters such that
  each pair of alternatives divides the order~$\triangleright^{\text{sc}}$ into at most two suborders, 
  where in each suborder, all voters have the same linear order on this pair. 
  We show that for $n$ voters and $n$ alternatives,
  the number of single-peaked narcissistic profiles is $\prod_{i=2}^{n-1} \binom{n-1}{i-1}$
  while the number of single-crossing narcissistic profiles is
  $2^{\binom{n-1}{2}}$.

\bigskip

\end{abstract}

\section{Introduction}\label{sec:intro}

We deal with permutations of an $n$-element set~$A\coloneqq \{1,2,\dots,n\}$ that satisfy some specific properties.
These properties arise from social choice theory, 
where each permutation is interpreted as the preference order of an individual on the set~$A$.
The elements of $A$ are called \emph{alternatives}.
In the following, we will first use terminology established in social choice,
and then introduce notions that are more commonly used in discrete mathematics.

Social choice theory, and voting theory in particular, 
deals with voters and their preferences on a set of alternatives. 
There, each voter~$i$ from a voter set~$\vvv$
has a \emph{preference order~$\pref_i$ on the set~$A$} (which is a linear order on $A$),
such that for each two alternatives~$a,b\in A$,
voter~$i$ \emph{prefers}~$a$ to~$b$ if $a \pref_i b$ holds.

When forming coalitions~\cite{Demange1994,BraJonKil2002}, 
building teams~\cite{BarTri1986,BreCheFinNie2017}, 
or playing games, the individuals, who we jointly denote as \emph{voters}, 
may have preferences over who is better than another as a potential coalition partner, 
a team member, or a player.
In such situations, the voters and alternatives are identical, that is, $A=\vvv$.
Deriving from a simple psychological model,
it seems natural to assume that 
each voter is \emph{narcissistic}~\cite{BarTri1986}, 
meaning that she is her own ideal and, thus, \emph{most preferred alternative}, that is, 
for each voter~$i\in \vvv$ and each alternative~$b\in \vvv\setminus \{i\}$, it holds that $i\pref_i b$.

Another well-studied property of voters preference orders on the set~$A$ of alternatives,
the \emph{single-peaked} property,
is characterized by a linear order~$\triangleright^{\text{sp}}$ of the alternatives,
where
for each voter~$i$,
her preferences along the order~$\triangleright^{\text{sp}}$ 
strictly increase until they reach the \emph{peak} which is her most preferred alternative,
and then strictly decrease, that is, 
for each alternative~$b\in A$, 
the set~$\{b\}\cup \{a \in A \mid a\pref_i b\}$ forms an interval in $\triangleright^{\text{sp}}$.
\citet{Black1948} introduced the concept of single-peakedness.
He observed that voters' political interests over the parties are single-peaked,
meaning that there is a left-to-right political spectrum of the parties 
such that each voter has a political ideal on this spectrum and the further away a party is from her ideal, 
the less she will like this party. 
Single-peaked preferences are also studied in psychology under the name \emph{unimodal orders}~\cite{Coombs1964,DoiFal1994}.
 
A third property, the \emph{single-crossing} property, 
requires that there is a linear order of the voters 
such that the preference orders of the voters on each pair of alternatives along this order 
change at most once, that is, 
there is a linear order~$\triangleright^{\text{sc}}$ of the voters
where for each two distinct alternatives~$a,b\in A$ and for each three distinct voters~$i,j,k\in \vvv$
with $i \triangleright^{\text{sc}}j\triangleright^{\text{sc}}k$,
if $a\pref_i b$ and $a\pref_k b$, then $a\pref_j b$. 
\citet{Mirrlees1971} and \citet{Roberts1977} introduced this concept in the field of economics.
They observed that voters' preferences on income taxation 
display a pattern that accords to their incomes,
and are thus single-crossing: 
When asked about the preferences over two tax rates~$x$ and $y$ with $x > y$,
if a voter~$v$ (the ``crossing'' spot) with medium income already changes from preferring~$x$ over~$y$
to preferring $y$ over~$x$,
then all voters with higher income than $v$ will also prefer $y$ over $x$.  
%
Single-crossingness goes back to the work of~\citet{Karlin1968} and is closely related to the partial ordered set on the set of all $n$-permutations, known as the \emph{weak Bruhat order}. We refer to the papers of~\citet{Abello1991,GaRe2008,BreCheWoe2013a} 
for more information.
See \cref{def:narcissistic,def:sp,def:sc} for a formal definition of the three properties we just introduced.

Research on restricted domains such as single-peaked or single-crossing preferences has been popular in  
political science, in psychology, in social choice, and quite recently in computational social choice.
We refer to the papers of \citet{BreCheWoe2016,ElkLacPet2016} for ample references to research on the two properties. 
Single-crossing preferences are not necessarily single-peaked, but
\citet{ST06} and \citet{BaMo2011} observed that single-crossing narcissistic preferences are single-peaked.
However, not all single-peaked narcissistic preferences are single-crossing.
For a simple illustration, the preferences of the following four voters are narcissistic.
\begin{quote}
\begin{tabular}{l@{}l@{}l@{}l@{}l@{}l@{}l@{}l@{}l@{}l@{}l@{}l}
    \defvoter~$v_1\colon~$ & $v_1$ & $~\succ_1~$ & $v_2$  & $~\succ_1~ $ & $v_3 $ & $~\succ_1~ $ & $v_4$,\\
    \defvoter~$v_2\colon~ $ & $v_2 $ & $~\succ_2~ $ & $v_3 $ & $~\succ_2~ $ & $v_4 $ & $~\succ_2~ $ & $v_1$,\\
    \defvoter~$v_3\colon~ $ & $v_3 $ & $~\succ_3~ $ & $v_2 $ & $~\succ_3~ $ & $v_4 $ & $~\succ_3~ $ & $v_1$,\\
    \defvoter~$v_4\colon~ $ & $v_4 $ & $~\succ_4~ $ & $v_3 $ & $~\succ_4~ $ & $v_2 $ & $~\succ_4~ $ & $v_1$.
\end{tabular}
\end{quote}

\noindent For instance, 
voter~$v_1$ is her most preferred alternative,
and $v_2$, $v_3$, and $v_4$ are voter~$v_1$'s
second most preferred, third most preferred, and least preferred alternative, respectively. 
These voter preferences are single-crossing, and also single-peaked, with respect to the order~$v_1\triangleright v_2 \triangleright v_3 \triangleright v_4$. See \cref{ex:not-sc-profile} for more information.

\noindent However, if we just swap the positions of $v_4$ and $v_1$ in the preference order of voter~$v_3$ to obtain 
\begin{quote}
\begin{tabular}{l@{}l@{}l@{}l@{}l@{}l@{}l@{}l@{}l@{}l@{}l@{}l}
  \defvoter~$v_3\colon~$ & $v_3$ & $~\succ_3~$ & $v_2$ & $~\succ_3~$ & $v_1$ & $~\succ_3~$ & $v_4$,
\end{tabular}
 \end{quote}

\noindent then the resulting voter preferences, together with voters~$v_1,v_2$, and $v_4$, are still single-peaked (with respect to the order~$\triangleright$) and narcissistic, but not single-crossing anymore. 
See \cref{ex:not-sc-profile} for further discussion.

In this work, we deal with preference profiles with $n$ voters who each have a preference order on all $n$ voters.
In general, there are $n!^n$ different preference profiles.
But how likely is it that one of these profiles will have some specific property?
For instance, the number of narcissistic profiles is $(n-1)!^n$.
So, one out of $n^n$ profiles is narcissistic.
\citet{LacLac2017} studied the likelihood of single-peaked preferences under some distribution assumption on the preference orders of the voters. 
However, we are interested in narcissistic profiles that are also single-peaked,
and that are also single-crossing.
More precisely, we investigate the numbers of narcissistic profiles that are also single-peaked (SPN), and of narcissistic profiles that are also single-crossing (SCN), respectively.
While it is quite straightforward to derive the number of SPN profiles,
this is not the case for SCN profiles.
Nonetheless, we are able to determine the number of SCN profiles with the help of \emph{semi-standard Young tableaux} (SSYT), by establishing a bijective relation between SSYTs and SCN profiles.

Our results are that for $n$ voters and $n$ alternatives,
the number of single-peaked narcissistic profiles is $\prod_{i=2}^{n-1} \binom{n-1}{i-1}$
while the number of single-crossing narcissistic profiles is
$2^{\binom{n-1}{2}}$.



\section{Basic Definitions and Fundamentals}\label{sec:basic}
In this section, we introduce basic terms from social choice~\cite[Chapter 4]{ArrSenSuz02}, combinatorics of permutations~\cite{Bona2004}, and Young tableaux~\cite{Stanley1999,Fulton1997,Yong2007}.

\subsection{Voters, alternatives, and preference orders}
Let $\vvv\coloneqq \{1,2,\ldots, n\}$ be a set of voters.
Since we are concerned with voters that have preferences over themselves,
$\vvv$~also plays the role of the set of alternatives.
A \emph{preference order}~$\pref$ on $\vvv$ is a strict linear order on~$\vvv$, 
that is, a binary relation on~$\vvv$ which is total, antisymmetric, and transitive.
Sometimes, we use the letters~$a,b,c,\ldots$ instead of the numbers~$1,2,\ldots$ to emphasize that we are considering the alternatives instead of the voters.
Given two disjoint subsets of alternatives~$A$ and $B$,
we use the notation~$A\pref B$
to express that a voter has a preference order~$\pref$ such that for each $a\in A$ and for each $b\in B$
it holds that $a\pref b$.
We simplify $A\pref B$ to $a\pref B$ if $A=\{a\}$ and $A\pref B$ to $A\pref b$ if $B=\{b\}$.

A preference profile~$\ppp(\vvv)$ of voter set~$\vvv$
is an $n$-tuple of preference orders for $\vvv$, that is,  $\ppp(\vvv) \coloneqq \prefs$,
where each $\pref_i$ represents the preference order of voter~$i$.
\begin{example}\label{ex:profile}
  If we rename the voters $v_i \mapsto i$ for all $i\in\{1, 2, 3, 4\}$ in the introductory example,
  then we obtain the following preference profile for the voter set~$\{1,2,3,4\}$:
  \begin{alignat*}{5}
    \defvoter~1\colon~ & 1 & ~\succ_1~ & 2 & ~\succ_1~ & 3 & ~\succ_1~ & 4,\\
    \defvoter~2\colon~ & 2 & ~\succ_2~ & 3 & ~\succ_2~ & 4 & ~\succ_2~ & 1,\\
    \defvoter~3\colon~ & 3 & ~\succ_3~ & 2 & ~\succ_3~ & 4 & ~\succ_3~ & 1,\\
    \defvoter~4\colon~ & 4 & ~\succ_4~ & 3 & ~\succ_4~ & 2 & ~\succ_4~ & 1.
\end{alignat*}
\end{example}

To describe the properties of preference profiles, 
for each preference order~$\pref$ and each subset of alternatives~$\vvv'\subseteq \vvv$,
we introduce the concept of \emph{top} alternatives~$i$ from $\vvv$ that are preferred over $\vvv'$.
 \text{For each preference order }$\pref$ \text{ and }  \text{for each subset } $\vvv'\subseteq \vvv$ \text{ of alternatives,}
 we define $\topp(\pref, \vvv')\coloneqq \{i \in \vvv \mid \forall j \in \vvv'\setminus \{i\} \text{ it holds that } i \pref j\}\text{.}$

For example, the top alternatives of preference order~$\pref_2$ with respect to $\{3,4\}$ are $2$ and $3$. 
Thus, $\topp(\pref_2, \{3,4\})=\{2,3\}$.

We use~$\peak(\pref)$ to denote the most preferred alternative in~$\pref$, that is, $\{\peak(\pref)\} \coloneqq \topp(\pref, \vvv)$.
We define the \emph{position of an alternative~$j$} in a preference order~$\pref$ in a common way, that is, 
it is one plus the number of alternatives that are preferred to her: $\pos(\pref, j)\coloneqq |\topp(\pref, \{j\})|\text{.}$


\subsection{Narcissistic profiles} \label{def:narcissistic}
We call a preference profile~$\ppp(\vvv)$ with voter set~$\vvv$ a \emph{narcissistic} profile if for each voter~$i\in \vvv$ it holds that she is her most preferred alternative, that is, 
 for each voter~$i\in \vvv$ it holds that $\peak(\pref_i) = i\text{.}$

\subsection{Single-peaked profiles} \label{def:sp}
Let $\ppp(\vvv)$ be a preference profile with voter set~$\vvv$,
 and let $\triangleright$ be a linear order on the set~$\vvv$.
We call a preference order~$\succ\in \ppp(\vvv)$ \emph{single-peaked with respect to~$\triangleright$} if for each two alternatives~$a,b\in \vvv$ it holds that
\[\text{if } a\triangleright b \triangleright \peak(\pref) \text{ or } \peak(\pref) \triangleright b \triangleright a\text{, then } b \pref a\text{.}\]
Accordingly, we call $\ppp(\vvv)$ \emph{single-peaked with respect to~$\triangleright$} if each preference order from $\ppp(\vvv)$ is single-peaked with respect to the order~$\triangleright$,  
and we call this order~$\triangleright$ a \emph{single-peaked order}~(for $\ppp(\vvv)$). 
As already mentioned in the introduction, 
\cref{ex:profile} is narcissistic and single-peaked with respect to the linear order~$1\triangleright 2 \triangleright 3 \triangleright 4$.

There are many equivalent definitions of the single-peaked property.
One of them is due to \citet{DoiFal1994}.
\begin{proposition}[\cite{DoiFal1994}]\label{prop:alternative-sp}
 Given a preference profile~$\ppp(\vvv)=\prefs$ and a linear order~$\triangleright$ on $\vvv$, the following statements are equivalent:
 \begin{compactenum}
   \item $\ppp(\vvv)$ is single-peaked with respect to $\triangleright$.
   \item For each voter~$i\in \vvv$, and for each alternative~$j\in \vvv$, 
   the top alternatives~$\topp(\pref_i, \{j\})$ form an interval in $\triangleright$.
 \end{compactenum}
\end{proposition}

\citet{DoiFal1994, EscLanOez2008} provided polynomial-time algorithms to determine whether a profile is single-peaked.
\citet{BaHa2011} characterized single-peaked profiles by two forbidden subprofiles:
\begin{proposition}[\cite{BaHa2011}]\label{prop:sp-characterization}
  A profile is single-peaked if and only if it contains
  neither a \emph{worst-subprofile} of three alternatives~$a, b, c$ and  three voters~$i, j, k$
  such that 
  \begin{align*}
    \defvoter~i\colon \{b,c\}\pref_i a\text{,} \quad  \defvoter~j\colon \{a,c\}\pref_j b\text{,}\quad    \defvoter~k\colon \{a,b\}\pref_k c\text{,}
  \end{align*}
  nor an $\alpha$-subprofile of four alternatives~$a, b, c, d$ and two voters~$i, j$
  such that 
    \begin{align*}
      \defvoter~i\colon \{a,b\} \pref_i c \pref_i d \text{, }  \defvoter~j\colon \{b,d\}\pref_j c \pref_j a\text{.}
    \end{align*}
\end{proposition}

\subsection{Single-crossing profiles} \label{def:sc}
Let $\ppp(\vvv)$ be a preference profile with voter set~$\vvv$,
and let $\triangleright$ be a linear order on the set~$\vvv$.
We call $\ppp(\vvv)$ \emph{single-crossing with respect to~$\triangleright$} if
for each pair~$\{a,b\}\subseteq \vvv$ of alternatives and for each three voters~$i, j, k \in \vvv$ with $i\triangleright j \triangleright k$, it holds that 
\[\text{if } a\pref_i b \text{ and } a\pref_k b \text{, then } a \pref_j b\text{.}\]
Accordingly, we call $\ppp(\vvv)$ a \emph{single-crossing profile} if there is 
a linear order~$\triangleright$ on the voter set~$\vvv$
with respect to which $\ppp(\vvv)$ is single-crossing,
and we call this order~$\triangleright$ a \emph{single-crossing order} (for $\ppp(\vvv)$).

Just as with single-peaked profiles, there are many equivalent definitions of the single-crossing property~\cite{DoiFal1994,BreCheWoe2013a,BreCheWoe2016}.
To introduce these alternative definitions, let $\diffpairs(\pref, \pref')$ denote the set of all pairs of alternatives that are ordered differently by $\pref$ and~$\pref'$:
\[
\diffpairs(\pref, \pref')\coloneqq \{\{i,j\}\subseteq \vvv\mid i\pref j \text{ and } j\pref' i\}\text{.}
\]

\begin{proposition}{\cite{DoiFal1994,BreCheWoe2013a,BreCheWoe2016}} 
 Given a preference profile~$\ppp(\vvv)=\prefs$ and a linear order~$\triangleright$ on $\vvv$, the following statements are equivalent:
 \begin{compactenum}
   \item $\ppp(\vvv)$ is single-crossing with respect to $\triangleright$.
   \item For each pair of alternatives~$\{a,b\}\subseteq \vvv$ 
   and for each two voters~$i,j\in \vvv$ with $i \triangleright j$, 
   it holds that $\diffpairs(\pref^*, \pref_i) \subseteq \diffpairs(\pref^*, \pref_j)$, 
   where $\pref^*$ denotes the preference order of the first voter in $\triangleright$.
   \item For each pair~$\{a,b\}\subseteq \vvv$ of alternatives,
   the voters that prefer~$a$ to $b$ form an interval in $\triangleright$ , 
   and the voters that prefer~$b$ to $a$ also form an interval in $\triangleright$, respectively.
 \end{compactenum}
\end{proposition}

\citet{DoiFal1994}, \citet{ElkFalSli2012}, and \citet{BreCheWoe2013a} provided polynomial-time algorithms to determine whether a profile is single-crossing.
\citet{BreCheWoe2013a} characterized single-crossing profiles by two forbidden subprofiles:
\begin{proposition}[\cite{BreCheWoe2013a}]\label{prop:sc-characterization}
  A profile is single-crossing if and only if it contains
  neither a $\gamma$-subprofile of three (not necessarily disjoint) pairs of alternatives~$\{a,b\}$, $\{c,d\}$, $\{e,f\}$ and three voters~$i, j, k$
  such that 
  \begin{alignat*}{4}
    \defvoter~i\colon~& a\pref_i b& ~\text{ and }~ & c \pref_i d &~\text{ and }~ & e\pref_i f\text{,}\\
    \defvoter~j\colon~& b\pref_j a& ~\text{ and }~ & d \pref_j c & ~\text{ and }~ & e\pref_j f\text{,}\\
    \defvoter~k\colon~& a\pref_k b& ~\text{ and }~ & d\pref_k c & ~\text{ and }~ & f\pref_k e\text{,}
  \end{alignat*}
  nor a $\delta$-subprofile of two (not necessarily disjoint) pairs of alternatives,
  $\{a, b\}$ and $\{c, d\}$,
  and four voters~$i, j, k, \ell$
  such that 
    \begin{alignat*}{3}
      \defvoter~i\colon~& a \pref_i b & ~\text{ and }~ & c \pref_i d \text{,}\\
      \defvoter~j\colon~& b \pref_j a  & ~\text{ and }~ & c \pref_j d \text{,}\\
      \defvoter~k\colon~& a \pref_k b  & ~\text{ and }~ & d \pref_k c \text{,}\\
      \defvoter~\ell\colon~& b \pref_{\ell} a  & ~\text{ and }~ & d \pref_{\ell} c \text{.}
    \end{alignat*}
\end{proposition}

\subsection{Fundamental observations}
As mentioned in the introduction,
\citet[Lemma 4]{ST06} showed that for narcissistic profiles, single-crossingnes implies single-peakedness (see \citet[Proposition 1]{BreCheFinNie2017} for another proof, which uses our terminology).

\begin{proposition}[\cite{ST06,BreCheFinNie2017}]\label{nsc->sp}
Each narcissistic profile that is single-crossing with respect to some linear order~$\rhd$
is also single-peaked with respect to the same order~$\rhd$.
\end{proposition}

However, not all single-peaked narcissistic profiles are single-crossing, as the following example shows.

\begin{example}\label{ex:not-sc-profile}
  The profile given in \cref{ex:profile} is narcissistic, single-peaked, and single-crossing with $1\triangleright 2 \triangleright 3 \triangleright 4$ 
  being the desired order for the single-peaked property and the single-crossing property.
  See \cref{fig:restricted-profiles} for an visualization of both properties.
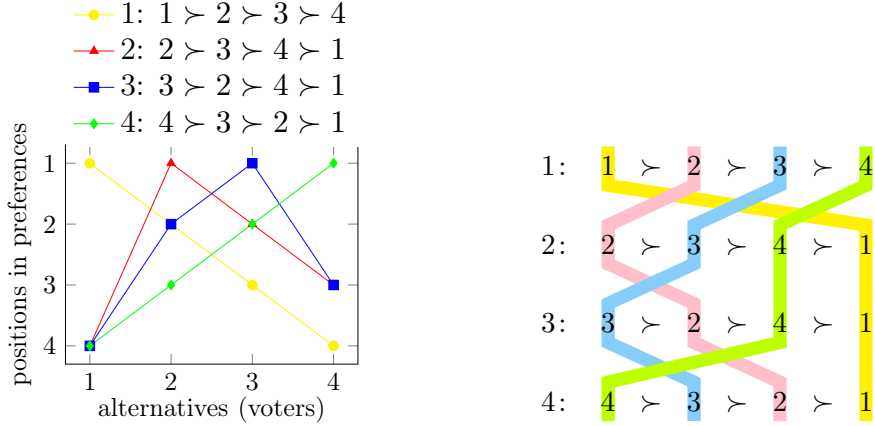
\begin{figure}[t]\captionsetup[subfigure]{position=t}
  \begin{subfigure}[b]{0.49\textwidth}
   \centering
    \begin{tikzpicture}[scale=0.9]
      \begin{axis}[
          symbolic x coords={1,2,3,4},
          xlabel shift=-1ex,
          xlabel=alternatives (voters),
          y dir=reverse,
          ylabel shift=-.5ex,
          ylabel=positions in preferences,
          xtick=data,
          ytick=data,
          x = 1.2cm,
          y = .9cm,
          legend style={draw=none,nodes={scale=1.2, transform shape},at={(0,1.7)},anchor=north west}
          ]
        \addplot[color=yellow,mark=*] coordinates {
          (1,  1)
          (2,  2)
          (3,  3)
          (4,  4)
        };

        \addplot[color=red,mark=triangle*] coordinates {
          (1,  4)
          (2,  1)
          (3,  2)
          (4,  3)
        };

        \addplot[color=blue,mark=square*] coordinates {
          (1,  4)
          (2,  2)
          (3,  1)
          (4,  3)
        };

        \addplot[color=green,mark=diamond*] coordinates {
          (1,  4)
          (2,  3)
          (3,  2)
          (4,  1)
        };
        \legend{\clR{1: $1 \succ 2 \succ 3 \succ 4$},\clB{2: $2 \succ 3 \succ 4 \succ 1$},\clG{3: $3 \succ 2 \succ 4 \succ 1$},\clO{4: $4 \succ 3 \succ 2 \succ 1$}}
      \end{axis}
    \end{tikzpicture}\end{subfigure}
  ~~
  \begin{subfigure}[b]{0.46\textwidth}
    \begin{tikzpicture}
  \tikzset{linemark/.style =   {line width= 5pt, color=#1}}

  \matrix (scprofilet) [matrix of math nodes, row sep=15pt, column sep=2pt] {
    1 \colon & 1 & \succ & 2  & \succ & 3 & \succ & 4\\
    2 \colon & 2 & \succ & 3 & \succ & 4 & \succ & 1\\
    3 \colon & 3 & \succ & 2 & \succ & 4 & \succ & 1\\
    4 \colon & 4 & \succ & 3 & \succ & 2 & \succ & 1\\
  };

  \begin{scope}[on background layer]
  \draw[linemark=yellow] 
  (scprofilet-1-2.north) -- (scprofilet-1-2.south) -- 
  (scprofilet-2-8.north) -- (scprofilet-2-8.south) -- 
  (scprofilet-3-8.north) -- (scprofilet-3-8.south) --
  (scprofilet-4-8.north) -- (scprofilet-4-8.south);

  \draw[linemark=lightrose] 
  (scprofilet-1-4.north) -- (scprofilet-1-4.south) -- 
  (scprofilet-2-2.north) -- (scprofilet-2-2.south) -- 
  (scprofilet-3-4.north) -- (scprofilet-3-4.south) --
  (scprofilet-4-6.north) -- (scprofilet-4-6.south);

  \draw[linemark=lightblue] 
  (scprofilet-1-6.north) -- (scprofilet-1-6.south) -- 
  (scprofilet-2-4.north) -- (scprofilet-2-4.south) -- 
  (scprofilet-3-2.north) -- (scprofilet-3-2.south) --
  (scprofilet-4-4.north) -- (scprofilet-4-4.south);

  \draw[linemark=lightgreen]
  (scprofilet-1-8.north) -- (scprofilet-1-8.south) -- 
  (scprofilet-2-6.north) -- (scprofilet-2-6.south) -- 
  (scprofilet-3-6.north) -- (scprofilet-3-6.south) --
  (scprofilet-4-2.north) -- (scprofilet-4-2.south);
  \end{scope} 
\end{tikzpicture}  \end{subfigure}
  \caption{Visualization of the single-peaked property (left) and the single-crossing property (right).
    Left: 
    The left-to-right order of the alternatives on the x-axis is a single-peaked order.
    The y-axis denotes the positions in a preference order.
    For each voter's preference order, we draw a colored line through the positions of all alternatives.
    In this way, we obtain a curve such that going from left to right on the x-axis, 
    the value (the position in the respective preference order) increases until it reaches its peak and then decreases.
    Right: 
    For each alternative, we draw a colored line, 
    which passes through the same alternative in each voter's preference order.
    It is easy to verify that each two colored lines cross at most once.
    This implies that the corresponding top-down order of the voters is a single-crossing order.}\label{fig:restricted-profiles}
\end{figure}
  However, if we change the preference order of voter~$3$ in \cref{ex:profile} to obtain the following
\begin{alignat*}{5}
    \defvoter~1\colon~ & 1 & ~\succ_1~ & 2 & ~\succ_1~ & 3 & ~\succ_1~ & 4,\\
    \defvoter~2\colon~ & 2 & ~\succ_2~ & 3 & ~\succ_2~ & 4 & ~\succ_2~ & 1,\\
    \defvoter~3\colon~ & 3 & ~\succ_3~ & 2 & ~\succ_3~ & 1 & ~\succ_3~ & 4,\\
    \defvoter~4\colon~ & 4 & ~\succ_4~ & 3 & ~\succ_4~ & 2 & ~\succ_4~ & 1,
\end{alignat*}
then the resulting profile is still narcissistic and single-peaked but not single-crossing anymore.
The reason is that the new profile contains subprofiles that are not single-crossing.
For instance, it contains a $\delta$-subprofile with respect to the pairs~$\{1,4\}$ and $\{2,3\}$, 
and the voters~$1,2,3,4$.
\end{example}

The above example and \cref{nsc->sp} show that 
single-crossing narcissistic profiles form a strict subset of single-peaked narcissistic profiles.

\subsection{Semi-standard Young tableaux}\label{subsec:SSYT}
For a positive integer~$n$, a \emph{semi-standard Young tableau (SSYT)} of order~$n$~\cite{Stanley1999}
consists of $n$ rows of positive integers that satisfy the following.
\begin{enumerate}[i)]
  \item For each~$i\in \{1,2,\dots,n\}$, 
  the $i^{\text{th}}$~row contains $n-i+1$ entries with integers between $1$ and $n$.
  \item When aligned in the upper-left corner (to obtain an isosceles right triangle),
  the entries weakly increase along each row and strictly increase down each column.
\end{enumerate}

\cref{ex:SSYT-3} illustrates how an SSYT looks like and shows all eight possible SSYTs of order three.

  \begin{example}    \label{ex:SSYT-3}
     \setlength{\tabcolsep}{2.5pt}
     \setlength{\extrarowheight}{1pt}
    There are eight different SSYTs of order~$3$:
     \begin{multicols}{8}
     \begin{tabular}{@{}|c|c|c|@{}}
       \hline
      $1$ & $1$ & $1$ \\
       \hline
      $2$ & $2$\\
       \cline{1-2}
       $3$\\
       \cline{1-1}
     \end{tabular}
     
     \smallskip

    \begin{tabular}{|l|l|l|}
       \hline
      $1$ & $1$ & $1$ \\
       \hline
      $2$ & $3$\\
       \cline{1-2}
       $3$\\
       \cline{1-1}
     \end{tabular}

     \begin{tabular}{|l|l|l|}
       \hline
      $1$ & $1$ & $2$ \\
       \hline
       $2$ & $2$\\
       \cline{1-2}
       $3$\\
       \cline{1-1}
     \end{tabular}
     
     \begin{tabular}{|l|l|l|}
       \hline
      $1$ & $1$ & $2$ \\
       \hline
      $2$ & $3$\\
       \cline{1-2}
       $3$\\
       \cline{1-1}
     \end{tabular}

     \begin{tabular}{|l|l|l|}
       \hline
      $1$ & $1$ & $3$ \\
       \hline
      $2$ & $2$\\
       \cline{1-2}
       $3$\\
       \cline{1-1}
     \end{tabular}

     \begin{tabular}{|l|l|l|}
       \hline
      $1$ & $1$ & $3$ \\
       \hline
      $2$ & $3$\\
       \cline{1-2}
       $3$\\
       \cline{1-1}
     \end{tabular}
        
     \begin{tabular}{|l|l|l|}
       \hline
      $1$ & $2$ & $2$ \\
       \hline
      $2$ & $3$\\
       \cline{1-2}
       $3$\\
       \cline{1-1}
     \end{tabular}
        
     \begin{tabular}{|l|l|l|}
       \hline
      $1$ & $2$ & $3$ \\
       \hline
      $2$ & $3$\\
       \cline{1-2}
       $3$\\
       \cline{1-1}
     \end{tabular}        
   \end{multicols}
 \end{example}

\paragraph{Remark} Young tableaux~\cite{Young1928} were originally defined on a Ferrers diagram which can be of an arbitrary staircase-like shape, that is, 
a Young tableau may contain $n$ rows of non-increasing lengths.
The numbers in a tableau can be from an arbitrary integer range.
For our purpose, it is sufficient to focus on SSYTs for isosceles right triangles and for integer range~$[1,n]$.
The second condition described above defines the ``semi-standard'' property.
We refer to the work of \citet{Stanley1999,Fulton1997,Yong2007} for further reading.

By \citet[Corollary 7.21.4]{Stanley1999}'s hook content lemma, 
we can deduce that the number of SSYTs of order~$n$, denoted as $\#_{\SSYT}(n)$, equals $2^{\binom{n}{2}}$.
Before we show this result, we need two more notions.
\begin{definition}[Hook lengths and hook contents]
  Let $n$ be an arbitrary positive integer. 
  For each two values~$i,j$ 
  with $1\le i \le n$ and $1\le j \le n+1-i$, we define the following two notions.
  \begin{enumerate}
    \item The \emph{hook length} of $(i,j)$ for the order~$n$, denoted as $h_n(i,j)$, is one plus the number of entries directly below or to the right of $T(i,j)$ in the $i^{th}$ row and $j^{th}$ column,
where $T$ is an arbitrary SSYT of order~$n$: $h_n(i,j)\coloneqq 2\cdot (n - i - j) + 3$.
  \item The \emph{hook content} of $(i,j)$ for the order~$n$ is defined as $c_n(i,j)\coloneqq n-i+j$.
  \end{enumerate}
\end{definition}

Note that the hook length and the hook content do not depend on the values of any SSYT but on its order.
Moreover, the original notion of \emph{(hook) content for $(i,j)$} as introduced by \citet[\S~7.21]{Stanley1999} equals $j-i$ instead of $n-i+j$. 
We use our definition of the hook content to simplify some of our equations below and the equation given in the hook content lemma~by \citet[Corollary~7.21.4]{Stanley1999}.

The following example shows the hook length and the hook content of all SSYTs shown in \cref{ex:SSYT-3}.

\begin{example}
   The hook lengths and hook contents of SSYTs of order $3$ are:

   \medskip

     \begin{tabular}{l|l|l|l|l}
      \cline{2-4}
      &$5$ & $3$ & $1$ \\
       \cline{2-4}
      Hook lengths &$3$ & $1$\\
       \cline{2-3}
      & $1$\\
       \cline{2-2}
     \end{tabular}
     ~~
      \begin{tabular}{l|l|l|l|l}
      \cline{2-4}
      &$3$ & $4$ & $5$ \\
       \cline{2-4}
      Hook contents &$2$ & $3$\\
       \cline{2-3}
      & $1$\\
       \cline{2-2}
     \end{tabular}

\end{example}

Now, we are ready to derive the number~$\#_{\SSYT}(n)$ SSYTs of order~$n$. 

\begin{theorem}\label{thm:ssyt}
  For each positive integer~$n$, the number~$\#_{\SSYT}(n)$ of semi-standard Young tableaux of order~$n$ equals $2^{\binom{n}{2}}$.
\end{theorem}
\begin{proof}
  \citet[Corollary 7.21.4]{Stanley1999}'s hook content formula states that
  the number~$\#_\SSYT(n)$ of SSYTs of order~$n$ equals
  \begin{align}\label{eq:hookcontentformula}
    \prod_{\substack{1\le i \le n\\ 1\le j \le n+1-i}}\frac{c_n(i,j)}{h_n(i,j)}.
  \end{align}
  Note that the product~\eqref{eq:hookcontentformula} is essentially the same as the one given in  \citet[Corollary 7.21.4]{Stanley1999} since we defined the hook content for $(i,j)$ to be $n-i+j$ instead of $j-i$.
  
We need to show that the product~\eqref{eq:hookcontentformula} equals $2^{\binom{n}{2}}$.
To accomplish this, we show the following:
  \begin{enumerate}[i)]
    \item $\#_{\SSYT}(1) = 1$ and 
    \item for each $n\ge 2$ it holds that $\#_{\SSYT}(n+1) = 2^{n}\cdot  \#_{\SSYT}(n)$.
  \end{enumerate}
  For the first equation, 
  we can easily check that there is only one SSYT of order one, 
  implying that $\#_{\SSYT}(1) = 1 = 2^{\binom{1}{2}}$ (note that the definition of binomial coefficients implies that $\binom{1}{2}=0$).
  
  Now, consider SSYTs of order~$n+1$.
  The definitions of hook lengths and hook contents imply the following equations:
\allowdisplaybreaks
  \begin{align}
  \nonumber \intertext{$\forall i\in \{2,3,\ldots, n+1\}, \forall j \in\{1,2,\ldots,n+2-i\} \colon$}
   \label{length}    h_{n+1}(i,j) & =~ 2\cdot (n+1 - i - j ) + 3 =  h_n(i-1,j)\text{,}\\
   \label{content}    c_{n+1}(i,j) & =~ n+1 - i +j =  c_n(i-1,j)\text{.}\\[-4ex]
   \nonumber   \intertext{$\forall j \in\{2,\ldots,n+1\} \colon$}
   \label{length2}    h_{n+1}(1,j) & =~ 2\cdot (n+1 - 1 - j ) + 3 =  h_n(1,j-1)\text{,}\\[-4ex]
   \nonumber   \intertext{$\forall j \in\{1,\ldots,n-1\} \colon$}
   \label{content2}    c_{n+1}(1,j) & =~ n+1 - 1 +j =  c_n(1,j+1)\text{.}
  \end{align}


  By the hook content formula,
  we can derive the number~$\#_{\SSYT}(n+1)$ from the hook lengths and hook contents of SSYTs of order~$n$:
  \begin{align}
    \nonumber   \#_{\SSYT}(n+1) & ~=  
    \prod_{\substack{1\le i \le n+1\\ 1\le j \le n+2-i}}\frac{c_{n+1}(i,j)}{h_{n+1}(i,j)}\\
    &~= \prod_{1\le j \le n+1}\frac{c_{n+1}(1,j)}{h_{n+1}(1,j)}\cdot \prod_{\substack{2\le i \le n+1\\ 1\le j \le n+2-i}}\frac{c_{n+1}(i,j)}{h_{n+1}(i,j)}\text{.}\label{ssyt-n+1}
 \end{align}

If we can show that the second factor and the first factor of the product on the right-hand side of equation~\eqref{ssyt-n+1} equal  $\#_{\SSYT}(n)$ and $2^{n}$, respectively, 
then by \eqref{ssyt-n+1}, we can derive that  $\#_{\SSYT}(n+1)=2^{n}\cdot \#_{\SSYT}(n)$.
Thus, it remains to show the following:
\begin{align}
\label{second-part}\prod_{\substack{2\le i \le n+1\\ 1\le j \le n+2-i}}\frac{c_{n+1}(i,j)}{h_{n+1}(i,j)} &= \#_{\SSYT}(n)\text{, and }\\
\label{first-part} \prod_{1\le j \le n+1}\frac{c_{n+1}(1,j)}{h_{n+1}(1,j)}&= 2^n \text{.}
\end{align}

\noindent To show the correctness of equation~\eqref{second-part}, we use equations~\eqref{length} and \eqref{content}:
\begin{align*}
  \prod_{\substack{2\le i \le n+1\\ 1\le j \le n+2-i}}\frac{c_{n+1}(i,j)}{h_{n+1}(i,j)}
  &~= \prod_{\substack{2\le i \le n+1\\ 1\le j \le n+2-i}}\frac{c_{n}(i-1,j)}{h_{n}(i-1,j)}\\
  &~= \prod_{\substack{1\le i' \le n\\ 1\le j \le n+1-i'}}\frac{c_{n}(i',j)}{h_{n}(i',j)} = \#_{\SSYT}(n)\text{.}
\end{align*}
The last equality holds by the hook content formula, as mentioned in the beginning of the proof. 

We show equation~\eqref{first-part} by induction on $n$.

First of all,
for $n=1$, we have that $\prod_{1\le j \le n+1}\frac{c_{n+1}(1,j)}{h_{n+1}(1,j)} = 
\frac{c_2(1,1)}{h_2(1,1)}\cdot \frac{c_2(1,2)}{h_2(1,2)} = \frac{2\cdot 3}{3 \cdot 1} = 2 = 2^1$.
Now, suppose that equation~\eqref{first-part} holds for~$n\coloneqq \ell-1$,
implying that $\prod_{1\le j \le \ell}\frac{c_{\ell}(1,j)}{h_{\ell}(1,j)}=~2^{\ell-1} \text{.}$
We show that the equality also holds for $n \coloneqq \ell$:
\begin{align*}
 \prod_{1\le j \le n+1}\frac{c_{n+1}(1,j)}{h_{n+1}(1,j)}
  &~~=~ 
     \prod_{1\le j \le \ell+1}\frac{c_{\ell+1}(1,j)}{h_{\ell+1}(1,j)}\\
  &~~=~  \frac{c_{\ell+1}(1,\ell+1)  \cdot c_{\ell+1}(1,\ell) }{h_{\ell+1}(1,1)} \cdot 
  \frac{\prod_{1\le j \le \ell-1}c_{\ell+1}(1, j)}{\prod_{2\le j \le \ell+1}h_{\ell+1}(1,j)}\\
  &~\overset{\text{def.}}{=}~  \frac{2\cdot \ell+1}{2\cdot \ell+1}\cdot
  c_{\ell+1}(1,\ell)
  \cdot
   \frac{\prod_{1\le j \le \ell-1}c_{\ell+1}(1, j)}{\prod_{2\le j \le \ell+1}h_{\ell+1}(1,j)}\\  
  &\overset{\eqref{length2}\eqref{content2}}{=} 
  c_{\ell+1}(1,\ell)
  \cdot \frac{\prod_{1\le j \le \ell-1}c_{\ell}(1, j+1)}{\prod_{2\le j \le \ell+1}h_{\ell}(1,j-1)}\\
   &~~=~  \frac{c_{\ell+1}(1,\ell)}{c_{\ell}(1,1)} \cdot 
   \frac{\prod_{0\le j \le \ell-1}c_{\ell}(1, j+1)}{\prod_{2\le j \le \ell+1}h_{\ell}(1,j-1)}\\
  &~\overset{\text{def.}}{=}~ 
  \frac{\ell+1 - 1 + \ell}{\ell-1+1} \cdot \frac{\prod_{1\le j' \le \ell}c_{\ell}(1, j')}{\prod_{1\le j' \le \ell}h_{\ell}(1,j')}\\
  &~~=~   2^{\ell} = 2^{n}\text{.}  
\end{align*}
\noindent The last equality holds by our induction assumption.
\end{proof}

\paragraph{Remarks} 
The proof of \cref{thm:ssyt} is rather crude and lengthy, but straightforward. 
Nevertheless, there  is another, shorter and more elegant proof, pointed out by one of the reviewers of the journal Discrete Mathematics, for deriving the desired number~$\#_{\SSYT}(n)$ by applying the Schur function on the integer partition~$(n,n-1,\dots,1)$~\cite[\S~7.15.1]{Stanley1999} and using the Vandermonde determinant identity.


\section{Counting Single-Peaked Narcissistic Profiles}\label{sec:nsp}
In this section, we study the number of single-peaked narcissistic (SPN) preference profiles for the voter set~$\vvv=\{1,2,\ldots, n\}$.
Recall that a voter is narcissistic if she ranks herself at the first position.
Thus,
\begin{align}
  \label{vi}
  \text{for each } i, \defvoter~i \text{ has preference order of the form } i \pref_i \ldots\text{.}
\end{align}

In the following, for SPN preference profiles with at least two voters,
there are always two voters whose preference orders are the reverse of each other.

\begin{lemma}\label{reverse-exists}
  For each single-peaked narcissistic profile~$\ppp=(\pref_1,\pref_2,\ldots, \pref_n)$,
  there are  two voters~$i,j$ such that $|\diffpairs(\pref_i, \pref_j)|=\binom{n}{2}$.
\end{lemma}


\begin{proof}
  Let $a_1 \rhd a_2 \rhd \dots \rhd a_n$ be a single-peaked order for the profile~$\ppp$.
  Then, by the narcissistic property (\ref{vi}), 
  $\peak(\pref_{a_1})=a_1$ and $\peak(\pref_{a_n})=a_n$.
  Following the single-peaked order $\rhd$,
  we obtain that the preference orders of $a_1$ and $a_n$ are $a_1 \pref_{a_1} a_2 \pref_{a_1}\dots \pref_{a_1} a_n$
  and $a_n \pref_{a_n} a_{n-1} \pref_{a_n} \dots \pref_{a_n} a_1$.
  This implies that $|\diffpairs(\pref_{a_1}, \pref_{a_n})|=\binom{n}{2}$.
\end{proof}

\begin{example}\label{ex:nsp-without-renaminig}
  Consider the following three voters~$1,2,3$ with preference orders:
  \begin{align*}
   \defvoter~1\colon 1 \pref_1 3 \pref_1 2\text{, }\;
    \defvoter~2\colon 2 \pref_2 1 \pref_2 3 \text{,\; and }
   \defvoter~3\colon 3 \pref_3 1 \pref_3 2\text{.}
    \end{align*}
  The profile~$(\pref_1,\pref_2,\pref_3)$ is narcissistic, and single-peaked with respect to the orders~$2\rhd 1 \rhd 3$ 
  and $3 \rhd 1 \rhd 2$.
  The preference orders of voters~$2$ and $3$ are reverse to each other.
\end{example}

By the proof of~\cref{reverse-exists}, we can rename the voters such that the preference orders of voter~$1$ and $n$ are the following.
\begin{align}
\label{v1+n}  \setlength{\tabcolsep}{2.5pt}
\begin{tabular}{lccccccc}
  \defvoter~$1\colon$ & $1$ & $\pref_1$& $2$ & $\pref_1$ & $\ldots$ & $\pref_1$ & $n$\text{,}\\
  \defvoter~$n\colon$ & $n$ & $\pref_n$ & $n-1$ & $\pref_n$ & $\ldots$ & $\pref_n$ & $1$\text{.}
\end{tabular}
\end{align}
It is easy to show that the only linear orders of alternatives with respect to which voters~$1$ and $n$ (and thus the whole profile) are single-peaked
must be the preference orders of either voter~$1$ or voter~$n$ (also see Lemma~5.1 of the work of \citet{ChePruWoe2017} for more details).

\begin{corollary}\label{lem:exactly-two-sp-orders}
  Each single-peaked narcissistic profile with at least two voters admits exactly two single-peaked orders.
\end{corollary}

Summarizing, by Statement~\eqref{v1+n} and by \cref{lem:exactly-two-sp-orders},  we can rename the voters such that
\begin{align}
   \text{each SPN profile has exactly two SP orders}\colon 1\triangleright 2 \triangleright \ldots \triangleright n \text{ and its reverse.}   \label{sp-order}
\end{align}

For SPN profiles that obey \eqref{sp-order}, we observe the following.

\begin{proposition}\label{prop:nsp-property}
  Let $\ppp(\vvv)$ be a single-peaked narcissistic preference profile with $n$~voters, 
  and let $a_1\triangleright a_2 \triangleright \dots \triangleright a_n$ be a single-peaked order for $\ppp(\vvv)$.
  Then, for each voter~$a_i\in \vvv$, the following holds:
  \begin{enumerate}
    \item the preference order of $a_i$ restricted to the set~$\{a_1,a_2,\dots,a_i\}$ is decreasing, that is, 
    $a_i\pref_{a_i}a_{i-1} \pref_{a_i} \dots \pref_{a_i} a_1$, and
    \item  the preference order of $i$ restricted to the set~$\{a_i,a_{i+1},\dots,a_n\}$ is increasing, that is, 
    $a_i\pref_{a_i} a_{i+1} \pref_{a_i} \dots \pref_{a_i} a_n$.
  \end{enumerate}
\end{proposition}

\begin{proof}
  Since $\ppp(\vvv)$ is narcissistic, 
  each alternative~$a_i$ is the peak~$\peak(\pref_{a_i})$ of her own preference order, 
  that is, 
  $\forall a_j\in \vvv\setminus \{a_i\}\colon a_i \pref_{a_i} a_j$.
  By the definition of single-peakedness,
  the two statements follow.
\end{proof}

\begin{example}  \label{ex:different-nsp}
  The profile given in  \cref{ex:nsp-without-renaminig} does not obey \eqref{v1+n},
  but it has exactly two single-peaked orders: $2\rhd 1 \rhd 3$ and its reverse.
  If we rename the voters in the profile given in \cref{ex:nsp-without-renaminig} to $1\mapsto 2$, $2\mapsto 1$ and $3\mapsto 3$,
  then we obtain the lower-right profile,
  which obeys  \eqref{v1+n}.
  In this case, we consider both profiles as the same SCN profile. 
\begin{multicols}{2}
  \noindent
  \begin{align*}
    \defvoter~1\colon 1\pref_1 2 \pref_1 3\text{,}\\
    \defvoter~2\colon 2\pref_2 1 \pref_2 3\text{,}\\
    \defvoter~3\colon 3\pref_3 2 \pref_3 1\text{.}
  \end{align*}

  \columnbreak

  \noindent
  \begin{align*}
    \defvoter~1\colon 1\pref_1 2 \pref_1 3\text{,}\\
    \defvoter~2\colon 2\pref_2 3 \pref_2 1\text{,}\\
    \defvoter~3\colon 3\pref_3 2 \pref_3 1\text{.}
  \end{align*}
\end{multicols}

Indeed, for three voters, we have two different SPN profiles (see above) obeying \eqref{v1+n}.
Both profiles are single-peaked with respect to the orders of voters~$1$ and $3$.
However, we consider both profiles as different although we can obtain the upper-right profile from the upper-left profile by renaming $1\mapsto 3$, $2\mapsto 2$ and $3\mapsto 1$.
\end{example}

We are interested in SPN profiles that obey \eqref{v1+n},
that is, SPN profiles that are single-peaked with respect to the order~$1\triangleright 2 \triangleright \dots \triangleright n$ and its reverse.
In the following, we show our main result for the number of SPN preference profiles.
\begin{theorem}\label{thm:nsp}
  The number of narcissistic profiles for $n$ voters ($n\ge 2$)
  that are single-peaked with respect to the order~$1\triangleright 2 \triangleright \dots \triangleright n$ 
  is $\prod_{2\le i \le n-1}\binom{n-1}{i-1}$.
\end{theorem}

\begin{proof}
  By \cref{prop:nsp-property}, 
  for each voter~$i$, $2\le i \le n-1$,
  her preference order must satisfy
  \begin{align*}
    \defvoter~i\colon i\pref_i i-1 \pref_i \cdots \pref_i 1 \text{ and } i \pref_i i+1 \pref_i \cdots \pref_i n\text{.}
  \end{align*}
  Thus, if it is clear which positions the alternatives $1,2,\ldots, i-1$ will occupy in the preference order of voter~$i$,
  then the positions of $i+1,i+2,\ldots, n$ are also clear.
  Then, the preference order of $i$ is also fixed.
  There are $\binom{n-1}{i-1}$ possible ways to give $i-1$ positions to alternatives~$1,2,\ldots, i-1$.
  Altogether, we obtain the desired result for the number of all different SPN profiles.
\end{proof}

\section{Counting Single-Crossing Narcissistic Profiles}\label{sec:nsc}
In this section, we study the number of single-crossing narcissistic (SCN) preference profiles for the voter set~$\vvv=\{1,2,\ldots, n\}$.
Just as in \cref{sec:nsp},
we are interested in the SCN profiles that are single-crossing with respect to the linear order~$1\triangleright 2 \triangleright \dots \triangleright n$.
Since SCN profiles are also SPN~(\cref{nsc->sp}), 
we obtain the following result.

\begin{proposition}\label{cor:prelim-sc}
  For each narcissistic profile with the voter set~$\vvv=\{1,2,\ldots, n\}$ that is single-crossing with respect to the order~$1\triangleright 2 \triangleright \cdots \triangleright n$,
  the following holds.
  \begin{enumerate}[i)]
    \item The profile is only single-peaked with respect to~$\triangleright$ and its reverse.
    \item For each $i\in \{1,2,\ldots, n\}$, 
    the preference orders of voter~$i$ restricted to~$\{1,2,$ $\dots, i\}$ and to $\{i,i+1,\dots, n\}$ are $i\pref_i i-1 \pref_i \dots \pref_i 1$ and $i\pref_i i+1 \pref_i \dots \pref_i n$, respectively.
    \item The positions of each alternative~$a$, $1\le a \le n-1$
    in the preference orders of the voters $a+1,a+2,\dots, n$ along the order $a+1\triangleright a+2 \triangleright \cdots \triangleright n$
    are non-decreasing.
  \end{enumerate}
\end{proposition}


\begin{proof}
  Since the reverse of $\triangleright$ is also a single-crossing order,
  by \cref{nsc->sp}, the first statement follows.
  The second statement follows from \cref{nsc->sp} and from the definition of single-peakedness.

  It remains to show the correctness of the last statement.   
 Towards a contradiction, suppose that the positions of an alternative~$a$ in the preference orders of voters~$a+1, a+2, \dots, n$ are \emph{not} non-decreasing.
  Then, there must be another alternative~$b$ and two voters~$i, j$,
  $a+1\le i < j \le n$, with preferences
  $b \pref_i a$ and $a\pref_j b$.
  But since voter~$a$ prefers $a\pref_a b$, we obtain that $\triangleright$ is not a single-crossing order because of the pair~$\{a,b\}$---a contradiction.
\end{proof}

The profile given in \cref{ex:not-sc-profile} is evidence that the number of SCN profiles is strictly less than the number of  SPN profiles.
But how many SCN profiles are there exactly?
To answer this question,
we first construct a function from SCN profiles with $n$ voters to SSYTs of order~$n-1$.
To this end, 
let $\mathbb{P}_n$ be the set of all SCN profiles with $n$ voters that 
are single-crossing with respect to the linear order~$1\triangleright 2 \triangleright \dots \triangleright n$,
and let $\mathbb{S}_{n-1}$ be the set of all SSYTs of order $n-1$.
\begin{definition}[A function from SCN profiles to SSYTs]\label{def:f}
  Define a function $f\colon \mathbb{P}_{n}\to \mathbb{S}_{n-1}$ 
  that maps each SCN profile~$\prefs\in \mathbb{P}_{n}$ for the voter set~$\vvv=\{1,2,\dots, n\}$ to an SSYT~$f(\prefs)=T$ of order~$n-1$ as follows.  

  For each alternative~$i$ except $n$ (that is,  $i\in \vvv\setminus \{n\}$),
  we construct the $i^{th}$~row of~$T$ with $n-i$ entries.
  Their values depend on the positions of alternative~$i$ in the preference orders of voters $n, n-1,\ldots, i+1$:
  \begin{align*}
    \forall j\in \{1,2,\ldots, n-i\}\colon T(i,j) \coloneqq n+1- \pos(\pref_{n+1-j}, i)\text{.} 
  \end{align*}
\end{definition}

Briefly put, the value of $T$ at the $i^{\text{th}}$ row and $j^{\text{th}}$ column equals 
the ``reverted'' position of alternative~$i$ in the preference order of voter~$n+1-j$.
The values of each column~$j$ are determined by the preference order of voter~$n+1-j$. 
\cref{tab:def-f} gives an illustration of how to build an SSYT~$T$ from a given preference profile.

\begin{table}
  \centering

\begin{tikzpicture}[every node/.style={text height=2ex, text depth=.5ex,anchor=base}, every row/.style={nodes={minimum height=.85cm, minimum width=1.6cm}}, column 2/.style={nodes={minimum width=2.9cm}}, column 3/.style={nodes={minimum width=2cm}}, column 4/.style={nodes={minimum width=2cm}}, column 6/.style={nodes={minimum width=2cm}}, column 5/.style={nodes={minimum width=2cm}}]
\matrix (m) [matrix of nodes,column sep=-\pgflinewidth, row sep=-\pgflinewidth]{
& |[draw,fill=gray!30]|$n+1-\pos(\cdot)$ & |[draw,fill=gray!30]|voter~$n$ & |[draw,fill=gray!30]|voter~$n-1$ & |[draw,fill=gray!30]|$\cdots$ & |[draw,fill=gray!30]|voter~$2$\\
& |[draw,fill=gray!30]|alternative~$1$&|[draw]| & |[draw]| & |[draw]| & |[draw]| \\
{\large $T\colon$}& |[draw,fill=gray!30]|alternative~$2$&|[draw]| & |[draw]| & |[draw]|  \\
& |[draw,fill=gray!30]|$\cdots$&|[draw]|$\cdots$ & |[draw]|$\cdots$   \\
& |[draw,fill=gray!30]|alternative~$n-1$&|[draw]|\\
};
\end{tikzpicture}
\caption{An illustration of constructing an SSYT~$T$ according to \cref{def:f}.}
\label{tab:def-f}
\end{table}

Note that we do not address the positions of alternative~$n$ (in any preference order)
since the positions of $1,2,\ldots, n-1$ determine the position of $n$.
Moreover,  the positions of all alternatives in the preference order of voter~$1$ are also fixed.
We use \cref{ex:profile} to illustrate our function given in \cref{def:f}.

\begin{example}   \label{ex:f}
    Let $\ppp$ denote the profile from \cref{ex:profile}. 
    By \cref{def:f}, the SSYT obtained for $\ppp$ is depicted in the figure below.
    
    {\centering
      \setlength{\extrarowheight}{1pt}
      
      \begin{tabular}{ l c | c | c |}
        \cline{2-4}
        &   \multicolumn{1}{|c|}{ $1$ }& $1$ & $1$ \\
        \cline{2-4}
        $T\colon$ & \multicolumn{1}{|c|}{ $2$ } & $3$\\
        \cline{2-3}
        & \multicolumn{1}{|c|}{  $3$ }\\
      \cline{2-2}
      \end{tabular}
      \par
    }

    \smallskip
    
    \noindent The positions of alternative~$2$ in the preference orders of voters~$3$ and $4$ are
    $2$ and $3$, respectively.
    Thus, the second row of $T=f(\ppp)$ has two entries: $T(2,1)=4+1-3=2$ and $T(2,2)=4+1-2=3$.
  
  Note that the preference order of the last voter is always fixed to   
  $n\pref_n n-1\pref_n \cdots \pref_n 1\text{.}$
  Indeed, the values of the first column in every SSYT are also fixed, namely $(1,2,\ldots, n-1)^{T}$.
  For order 3, there are eight such SCN profiles.
  Our function~$f$ will map each of these profiles to a unique SSYT of order~$3$ given in \cref{ex:SSYT-3}.
\end{example}

In the following, we show that function~$f$ is well-defined and bijective.

\begin{lemma}\label{lem:f-well-defined}
  Function~$f$ from \cref{def:f} is well-defined.
\end{lemma}

\begin{proof}
  To show that $f$ is well-defined, we need to show that for each given SCN profile~$P=\prefs$ with $n$ voters,
  $f(P)$ is an SSYT of order~$n-1$.
  That is, we have to show that $T\coloneqq f(P)$ fulfills the two conditions given in the beginning of \cref{subsec:SSYT}.

  By \cref{def:f}, 
  $f(P)$ has $n-1$ rows such that for each value~$i$, $1 \le i \le n-1$, the $i^{\text{th}}$ row has $(n-1)+1-i$ entries.
  Moreover, for each alternative~$i\in \{1,2,\ldots, n-1\}$ and each voter~$j\in \{1,2,\ldots, n-i\}$,
  the value~$n+1-j$ ranges from $n$ to~$i+1$.
  Thus, the position~$\pos(\pref_{n+1-j}, i)$ of alternative~$i$ in the preference order of voter~$n+1-j$ is defined
  and, by the narcissistic property, has a value between $2$ and $n$.
  This means that the value of $T(i,j)$,
  which is defined as $n+1-\pos(\pref_{n+1-j}, i)$,
  is between $n-1$ and $1$.
  
  Second, by the last statement in \cref{cor:prelim-sc},
  the positions of the alternative~$i\in \{1,2,\ldots, n-1\}$
  in the preference orders of the voters are non-decreasing along the voter order~$i+1 \triangleright i+2 \triangleright \cdots \triangleright n$.
  By the double negation in the definition of $T(i,j)$, 
  this implies that the values along the $i^{\text{th}}$ row in $T$ do not decrease.

  It remains to show that the values down each column in $T$ strictly increase.
  The entries in each column~$j$ reflect the positions of the alternatives~$1$ to~$n-j$ 
  in the preference order of voter~$n+1-j$.
  By the second statement in \cref{cor:prelim-sc}, 
  it follows that the positions of the alternatives $1$ to $n-j$ strictly decrease and by the negation in the definition of~$T(i,j)$,
  we have that the values down each column in $T$ indeed strictly increase.
\end{proof}

\begin{lemma}\label{lem:f-bijective}
   Function~$f$ from \cref{def:f} is bijective.
\end{lemma}
\begin{proof}
  To show that $f$ is injective,
  consider two arbitrary SCN preference profiles~$\ppp=(\pref_1,\pref_2,\dots, \pref_n)$ 
  and $\ppp'=(\pref'_1,\pref'_2,\dots,\pref'_n)\in \mathbb{P}_{n}$
  that are single-crossing with respect to the linear order~$1\triangleright 2 \triangleright \dots \triangleright n$
  such that $f(\ppp)=f(\ppp')$.
  This means that for each column~$j\in \{1,2,\dots, n-1\}$ and each row~$i\in \{1,2,\dots, n-j\}$,
  we have that $f(\ppp)(i,j) = f(\ppp')(i,j)$,
  meaning that the position~$\pos(\pref_{n+1-j}, i) $ of each alternative~$i$ in the preference order~$\pref_{n+1-j}$ is the same as that~$\pos(\pref'_{n+1-j}, i)$ in the preference order~$\pref'_{n+1-j}$.
  By the single-peaked and narcissistic property, the preference order of each voter~$n+1-j$ is determined by the positions of the alternatives~$1$ to $n-j$. 
  Thus, the preference orders~$\pref_{n+1-j}$ are the same as the preference order~$\pref'_{n+1-j}$.
  Since the first voter always has the preference order~$1\pref 2\pref \dots \pref n$ in all profiles of $\mathbb{P}_n$, 
  we obtain that  $\ppp=\ppp'$.

  It remains to show that $f$ is surjective.
  For each SSYT~$T$ of order~$n-1$,
  there is an SCN preference profile~$\ppp=\prefs$ with the following form:
  \begin{align*}
    \text{voter~}1 &\colon 1 \pref_1 2 \pref_1  \cdots \pref_1 n\text{,} \\
    \text{voter~}2 &\colon 2 \pref_2 \cdots\text{,} \\
    \cdots\\
    \text{voter~}n &\colon n \pref_n n-1 \pref_n \cdots \pref_n 1\text{.}
  \end{align*}

  \noindent Formally, the first voter has preference order $1\pref_1 2\pref_1 \cdots \pref_1 n$, 
  and for each voter~$i\in \{2,\ldots, n\}$, 
  we first let $\pos(\pref_i,i)=1$ and then define
  her preference order~$\pref_i$ by defining the positions of the alternative~$j \in \{i-1,i-2,\ldots, 1\}$:
  $\pos(\pref_i, j)\coloneqq n+1 - T(j,n+1-i)$, which is at least two;
  observe that by the definition of SSYT, we have that $i \pref_i i-1 \pref_i \cdots \pref_i 1$.
  The remaining positions in the preference order~$\pref_i$ are assigned to the remaining alternatives $i+1, i+2,\ldots, n$ 
  such that $i+1 \pref_i i+2 \pref_i \dots \pref_i n$.
  Note that no two positions~$\pos(\pref_{i},j)$ and $\pos(\pref_i, j')$ with $1\le j < j'\le i$ are the same as no two entries in a column in $T$ have the same values,
  which means that we indeed obtain a preference order.
  
  Now, we show by contradiction that the constructed profile~$\ppp$ is single-crossing with respect to the order~$1 \rhd 2 \rhd \dots \rhd n$.
  Suppose, for the sake of contradiction, that there are two alternatives~$i$ and $j$ and three voters~$a, b$, and $c$ with $a < b < c$
  such that 
  \begin{alignat}{2}
    \label{pref:a} &\defvoter~a\colon &&i \pref_a j \text{,}\\
    \label{pref:b} &\defvoter~b\colon& &j\pref_b i\text{, and}\\
    \label{pref:c} &\defvoter~c \colon &&i \pref_c j\text{.}
  \end{alignat}
  First of all, we show the following three auxiliary statements which will be used many times in our proof:
  
  \begin{claim}\label{claim:1}
    Let $x$ and $y$ be two distinct alternatives from $\{1,2,\dots,n\}$ such that $x<y$, 
    and let $p$ be a voter from $\{1,2,\dots,n\}$ with preference order~$\pref_p$ from the constructed preference profile~$\ppp$.
    Then, the following holds.
    \begin{enumerate}[(i)]
      \item\label{claim:11} If $y \le p$, then $\pos(\pref_p,x)>\pos(\pref_p,y)$.
      \item\label{claim:12} If $x \ge p$, then $\pos(\pref_p,x)<\pos(\pref_p,y)$.
    \end{enumerate}
  \end{claim}
  
  \begin{proof}[Proof of \cref{claim:1}] \renewcommand{\qedsymbol}{(of
      \cref{claim:1})~$\diamond$} 
    Assume that $y < p$. Then, from the definition of $\ppp$, 
    we know that $T(x,n+1-p)$ and $T(y,n+1-p)$ are defined and
    the property of SSYT implies that $T(x,n+1-p)<T(y,n+1-p)$ since $x<y$.
    By the definition of the positions~$\pos$,
    we immediately have that $\pos(\pref_p,x)>\pos(\pref_p,y)$.
    If $y=p$, then $\pos(\pref_p,y) = 1 <2 \le \pos(\pref_p,x)$ .
    Together, we showed the first statement.

    If $x\ge p$, then by the definition of the positions of all $j$ with $j\ge p$ in $\pref_p$,
    we have $\pos(\pref_p,x)<\pos(\pref_p,y)$.
  \end{proof}

  \begin{claim}\label{claim:2}
    Let $x$ and $y$ be two distinct alternatives from $\{1,2,\dots,n\}$  such that 
    $x<y$, 
    and let $p$ be a voter from $\{1,2,\dots,n\}$ with preference order~$\pref_p$ from the constructed preference profile~$\ppp$,
    such that $x\le p < y$.
    The following holds.
    \begin{enumerate}[(i)]
      \item\label{claim:21} If $y\pref_p x$, then $\pos(\pref_p,x) > y-x$.
      \item\label{claim:22} If $x\pref_p y$, then $\pos(\pref_p,x)\le y-x$.
    \end{enumerate}
  \end{claim}
  
  \begin{proof}[Proof of \cref{claim:2}] \renewcommand{\qedsymbol}{(of
      \cref{claim:2})~$\diamond$} 
    Both statements obviously hold for $p=x$ (this includes the case of $p=1$), 
    since $x=p\pref_p y$ (each voter is narcissistic), 
    implying that $\pos(\pref_p,x)=1\le y-x$.

    In the remainder of the proof, we assume that $p\in \{2,3,\dots, n\}$ and $x< p$.
    First of all, the relation~$x < p$ implies that for each integer~$x'$ with $x+1 \le x' \le p-1$, the entry~$T(x',n+1-p)$ is defined. 
    To compute the value of~$\pos(\pref_p,x)$,
    we partition the set~$\{x+1,x+2,\dots,y\}$ into two disjoint subsets~$V_1\uplus V_2$
    with  $V_1\coloneqq \{x+1,\ldots, p-1\}$ and $V_2\coloneqq \{p,p+1,\dots,y\}$.
    We know that for every alternative~$x'$ with $x' \in V_1$,
    the value~$T(x',n+1-p)$ is defined, and, 
    by the property that the entries of each column in $T$ are strictly increasing, 
    for each $x'\in V_1$, it holds that
    \begin{align}\label{eq:V_1}
      T(x',n+1-p)>T(x,n+1-p) \text{, implying } \pos(\pref_p,x') < \pos(\pref_p,x)\text{.}
    \end{align}
    
    To show the first statement, assume that $y \pref_p x$.
    Then, by the definition of the preference order of voter~$p$, 
    for each alternative~$y'\in V_2$, it holds that $\pos(\pref_p,y')\le \pos(\pref_p,y)<\pos(\pref_p,x)$.
    Note that $\pos(\pref_p,y)<\pos(\pref_p,x)$ holds by our assumption that $y \pref_p x$.
    Together with the positions of all alternatives from $V_1$ (see \eqref{eq:V_1}), 
    there are at least $|V_1|+|V_2|$~alternatives preferred to $x$ by voter~$p$; therefore,
    we have $\pos(\pref_p,x) > |V_1|+|V_2|=y-x$.

    To show the second statement, assume that $x\pref_p y$. 
    Then, by the definition of the preference order of voter~$p$,
    we have the following.
    \begin{enumerate}
      \item for each alternative~$x'\in \{1,2,\dots, x-1\}$, it holds that $\pos(\pref_p, x') > \pos(\pref_p, x)$, and
      \item for each alternative~$y'\in \{y,y+1,\dots, n\}$, it holds that $\pos(\pref_p, y') \ge \pos(\pref_p, y) > \pos(\pref_p,x)$ (note that the second inequality holds by our assumption that $x\pref_p y$).
    \end{enumerate}
    This implies that at least $x-1+n-y+1=n+x-y$ alternatives have a larger position than alternative~$x$, 
    therefore $\pos(\pref_p, x)\le n-(n+x-y)=y-x$.
  \end{proof}

  \begin{claim}\label{claim:3}
    Let $x$ be an alternative from $\{1,2,\dots,n\}$ and let $p,q$ be two distinct voters from~$\{1,2,\dots,n\}$ such that $x<p<q$.
    Then, $\pos(\pref_p,x) \le \pos(\pref_q,x)$.
  \end{claim}
   \begin{proof}[Proof of \cref{claim:3}] \renewcommand{\qedsymbol}{(of
      \cref{claim:3})~$\diamond$} 
    Since $x<p<q$, the values~$T(x,n+1-p)$ and $T(x,n+1-q)$ are defined.
    By the weakly increasing property of each row in $T$, 
    the relation~$p<q$ implies that $\pos(\pref_p,x)=n+1-T(x,n+1-p) \le n+1-T(x,n+1-q)=\pos(\pref_q,x)$. 
  \end{proof}

  Now, we move on to our proof with case distinction on the relation between $i$ and $j$.

  Assume that $i<j$, and define $x=i$ and $y=j$.
  Then, \cref{claim:1}(\ref{claim:12}) and \eqref{pref:b} imply $i<b$ (applying $p=b$),
  and \cref{claim:1}(\ref{claim:11}) and \eqref{pref:c} imply $j>c$ (applying $p=c$).
  This means that $i<b<j$ and $i<c<j$ since $i<b<c<j$.
  If we use \cref{claim:2}(\ref{claim:21}) for the preference order given by \eqref{pref:b}, 
  then applying $p=b$,
  we have $\pos(\pref_b,i) > j-i$.
  However, if we use \cref{claim:2}(\ref{claim:22}) for the preference order given by \eqref{pref:c},
  then by applying $p=c$, we have $\pos(\pref_c,i)\le j-i$.
  This implies that $\pos(\pref_b,i)>\pos(\pref_c,i)$---a contradiction to \cref{claim:3} since $i<b<c$.
  
  Analogously,  $j<i$ yields a contradiction when we consider the preference orders~\eqref{pref:a} and \eqref{pref:b} instead.
  Define $x=j$ and $y=i$.
  \cref{claim:1}(\ref{claim:12}) and \eqref{pref:a} imply $j<a$ (applying $p=a$),
  and 
  \cref{claim:1}(\ref{claim:11}) and \eqref{pref:b} imply $i>b$ (applying $p=b$).
  This means that $j<a<i$ and $j<b<i$ since $j<a<b<i$.
  If we use \cref{claim:2}(\ref{claim:21}) for the preference order given by \eqref{pref:a}, 
  then by applying $p=a$,
  we have $\pos(\pref_a,j) > i-j$.
  However, if we use \cref{claim:2}(\ref{claim:22}) for the preference order given by \eqref{pref:b},
  then by applying $p=b$, 
  we have $\pos(\pref_b,j)\le i-j$.
  This implies that $\pos(\pref_a,j)>\pos(\pref_b,j)$---a contradiction to \cref{claim:3} since $j<a<b$.

  Summarizing, we showed that $1\rhd 2\rhd \dots \rhd n$ is indeed a single-crossing order.
  This implies that $f$ is also surjective, and thus bijective.
\end{proof}

Applying the inverse of function~$f$ given in \cref{def:f} on the SSYT produced in \cref{ex:f}
and assigning the remaining positions to the other remaining alternatives according to the proof of \cref{lem:f-bijective},
we will obtain our original profile from \cref{ex:profile}.
By \cref{thm:ssyt} and \cref{lem:f-well-defined,lem:f-bijective}, 
we obtain our second main result.

\begin{theorem}\label{thm:nsc}
  The number of narcissistic profiles for $n$ voters ($n\ge 2$)
  that are single-crossing with respect to the order~$1\triangleright 2 \triangleright \dots \triangleright n$ 
  is $2^{\binom{n-1}{2}}$.
\end{theorem}

\begin{proof}
  Let $\mathbb{P}_n$ be the set of all SCN profiles with $n$ voters that 
  are single-crossing with respect to the order~$1\triangleright 2 \triangleright \dots \triangleright n$, 
  and let $\mathbb{S}_{n-1}$ be the set of all SSYTs of order $n-1$.
  
  It is clear that both~$\mathbb{P}_n$ and $\mathbb{S}_{n-1}$ are finite.
  Since \cref{def:f} defines a function $f\colon \mathbb{P}_n \to \mathbb{S}_{n-1}$ that is a bijection (see \cref{lem:f-well-defined,lem:f-bijective}), 
  $\mathbb{P}_n$ and $\mathbb{S}_{n-1}$ have the same cardinality.
  By \cref{thm:ssyt}, we obtain the desired cardinality for~$\mathbb{P}_n$.
\end{proof}

\section{Conclusion}
We studied the numbers of narcissistic profiles that are also single-peaked (SPN) or also single-crossing (SCN), respectively.
We established a bijective relation between semi-standard Young tableaux and SCN profiles. 
By counting the number of semi-standard Young tableaux, 
we determined the number of SCN profiles.
In this paper, we focused on profiles with the same number of voters and alternatives .
However, our analysis could be extended to the case where the number of voters is greater than the number of alternatives
since the last statement of \cref{cor:prelim-sc} still holds in this case and it corresponds to the essential property of an SSYT for an arbitrary Ferrers diagram.

We focused on profiles that are single-peaked or single-crossing with respect to the linear order~$\triangleright\colon 1\triangleright 2 \triangleright \dots \triangleright n$.
An interesting question is to count the SPN (resp.\ SCN) profiles that are different up to renaming.
Herein, two profiles are said to be the same if one can be obtained from the other by renaming the voters. 
It is clear that under such restriction, the number is significantly smaller than the one we studied in the current paper.
While it seems quite straightforward to obtain the result for SPN profiles that are unique up to renaming,
this is not the case for SCN profiles.
We consider the corresponding study to be a part of future research.

\section*{Acknowledgments}

We thank the anonymous reviewers from Discrete Mathematics for pointing out a shorter proof of \cref{thm:ssyt} and for improving the presentation of the paper.

We thank Robert Bredereck (TU Berlin, Germany) for initial discussion on this project and Laurent Bulteau (Laboratoire d'Informatique Gaspard Monge in Marne-la-Vall{\'e}e, France) for some important references while Jiehua Chen was visiting him in March 2016; the visit was funded by Laboratoire d'Informatique Gaspard Monge in Marne-la-Vall{\'e}e, France.

The main work was done while Jiehua Chen was with TU Berlin, Germany. 
While with Ben-Gurion University of the Negev, she was supported by the People Programme (Marie Curie Actions) of the European Union's Seventh Framework Programme (FP7/2007-2013) under REA grant agreement number~631163.11, 
and by the
Israel Science Foundation (grant no.\ 551145/14).

\bibliographystyle{abbrvnat}

\newcommand{\bibremark}[1]{\marginpar{\tiny\bf#1}}

\end{document}